\documentclass[12pt,reqno]{amsart}

\usepackage[margin=1in]{geometry}
\usepackage{graphicx}
\usepackage{braket}
\usepackage{color}
\usepackage{geometry}
\usepackage{marginnote}
\usepackage{hyperref}

\usepackage{amssymb}

\DeclareMathOperator{\Ran}{Ran}
\DeclareMathOperator{\Dom}{Dom}

\DeclareMathOperator{\Ker}{Ker}

\DeclareMathOperator{\sign}{sgn}

\newtheorem{theorem}{Theorem}[section]
\newtheorem{lemma}[theorem]{Lemma}

\newtheorem{definition}[theorem]{Definition}

\theoremstyle{remark}
\newtheorem{remark}[theorem]{Remark}
\newtheorem{example}[theorem]{Example}

\newcommand\Fo{F^\circ}  
\newcommand\lo{\lambda^\circ}
\newcommand\cH{{\mathcal{H}}}
\newcommand\cM{{\mathcal{M}}}
\newcommand\cK{{\mathcal{K}}}

\newcommand\cL{{\mathcal{L}(\cH, \cK)}}
\newcommand\Id{{\mathrm{I}}}
\newcommand\spess{\sigma_{\mathrm{ess}}}

\newcommand\dK{\delta\!K}

\newcommand\cV{{\mathcal{V}}}
\newcommand\cC{{\mathcal{C}}}
\newcommand\cE{{\mathcal{E}}}
\newcommand\R{\mathbb{R}}
\newcommand\C{\mathbb{C}}
\newcommand\cN{{\mathcal{N}}}

\newcommand{\term}[1]{\textbf{#1}}

\title{Spectral shift via ``lateral'' perturbation}

\author{G.~Berkolaiko}
\author{P.~Kuchment}
\dedicatory{Dedicated to the memory of Misha Shubin, a wonderful
  mathematician and person}

\begin{document}
\date{\today}

\begin{abstract}
  We consider a compact perturbation $H_0 = S + K_0^* K_0$ of a
  self-adjoint operator $S$ with an eigenvalue $\lambda^\circ$ below
  its essential spectrum and the corresponding eigenfunction $f$.  The
  perturbation is assumed to be ``along'' the eigenfunction $f$,
  namely $K_0f=0$.  The eigenvalue $\lambda^\circ$ belongs to the
  spectra of both $H_0$ and $S$.  Let $S$ have $\sigma$ more
  eigenvalues below $\lambda^\circ$ than $H_0$; $\sigma$ is known as
  the spectral shift at $\lambda^\circ$.

  We now allow the perturbation to vary in a suitable operator space
  and study the continuation of the eigenvalue $\lambda^\circ$ in the
  spectrum of $H(K)=S + K^* K$.  We show that the eigenvalue as a
  function of $K$ has a critical point at $K=K_0$ and the Morse index
  of this critical point is the spectral shift $\sigma$. A version of this theorem also holds for some non-positive perturbations.
\end{abstract}

\maketitle

\section*{Introduction}\label{S:Intro}

The first step in the proofs of several spectral geometry theorems is
perturbing the operator ``along'' a given eigenfunction $f$.
To give a classical example, the Courant bound on the number of nodal
domains of the $n$-th eigenfunction $f=f_n$ of a Dirichlet Laplacian is
shown by introducing additional Dirichlet conditions along the zero
set of $f$.  The function $f$ is still an eigenfunction of the
perturbed operator and, as a consequence, the corresponding eigenvalue
$\lambda$ remains in the spectrum.

Recently, it was discovered that some nodal properties of eigenfunctions are related to stability with respect to perturbation of the original operator of suitably defined energy functionals.
More precisely, the nodal deficiency of the $n$-th eigenfunction $f_n$
on a manifold (defined as $n$ minus the number of the nodal domains of
$f_n$) is equal to the Morse index of the energy of the nodal
partition with respect to variation of the partition boundaries
\cite{BerKucSmi_gafa12}. On graphs, the nodal surplus (defined as the
number of zeros of $f_n$ minus $n-1$) is equal to the Morse index of
$\lambda_n$ considered as a function of the perturbation of the
Schr\"odinger operator by the magnetic field
\cite{Ber_apde13,Col_apde13,BerWey_ptrsa14}.  One is left wondering
what other types of perturbations can produce similar results. The
answer is presented in this paper. Essentially this is true for
\emph{any} ``sufficiently rich'' family of perturbations.

At this point, we set up notation and outline terms and conditions.
Let $\cH$ be a separable Hilbert space with the inner product
$\langle\cdot,\cdot\rangle$ (assumed linear with respect to the second
argument) and $S:\cH \to \cH$ be a self-adjoint operator bounded from
below.  Assume that below its essential spectrum, $S$ has an
eigenvalue $\lo$ with the eigenfunction $f$.  Consider further a
self-adjoint non-negative\footnote{Sign-indefinite perturbations will
  also be considered in the paper.  Here, for simplicity, we assume
  non-negativity.} perturbation operator $P$ such that $P f = 0$.
This is a perturbation ``along'' the eigenfunction $f$: $f$ is also an
eigenfunction of the perturbed operator $H := S + P$ with eigenvalue
$\lo$.  Assume that $\lo$ is simple in the spectrum of $S+P$.  If
$\lo$ has index $n$ in the spectrum of $H$, i.e.
$\lo = \lambda_n(S+P_0)$ then, due to positivity of $P_0$,
$\lo = \lambda_{n+\sigma}(S)$ with some integer $\sigma \geq 0$.  We
call this value $\sigma$ the \term{spectral shift}.  In the special
case when $P_0$ has rank $r < \infty$, one has $0 \leq \sigma \leq r$.

We remark that, in the hindsight, the theorems about nodal surplus or
deficiency mentioned above are in fact statements about the spectral
shift followed by some known relation between the index of the
eigenvalue and the nodal count for the perturbed operator $H$.  The spectral shift $\sigma$ and its relations to Morse indices is the primary object of interest here.

We now represent $P$ as a product\footnote{This is a positive perturbation, however more general perturbations are treated in the main result.} $P= K_0^* K_0$, where $K_0$ is
a compact operator from $\cH$ to an auxiliary Hilbert space $\cK$ and
$K_0 f = 0$.  We now allow the operator $K_0$ to vary, and consider
the continuation of the eigenvalue $\lo$ as a function of $K$.
Namely, we consider $\Lambda(K) := \lambda(A + K^* K)$ such that
$\Lambda(K_0) = \lo$.  Due to the standard perturbation theory,
this function is (real-)analytic with respect to $K$.  We will prove
that $\Lambda(K)$ has a critical point at $K = K_0$ and, if the family
of variations $K$ is ``rich enough,'' the Morse index of this critical point
is equal to the spectral shift $\sigma$. Here the \term{Morse index} is the number of negative eigenvalues of the Hessian at the critical point of the function.

By ``rich enough'' we mean the following.  Perturbations by operators
annihilating $f$ preserve $f$ as eigenfunction and do not affect the
eigenvalue.  We are interested in further (``\term{lateral}'')
perturbations, which do change the eigenvalue and carry information
about the spectral shift.  To capture the entirety of this information
(in the form of the Morse index), the family of variations has to be
transversal to the subspace of operators $K$ such that $Kf = 0$.

This result is important for a variety of extremal eigenvalue problems.  For example,
the question of optimizing an eigenvalue with respect to the
\emph{location} of a given perturbation has direct relevance to many
applications, such as, for instance, photonic crystals (where one is
interested in impurity modes in spectral gaps, to confine photons in
cavities), or civil engineering (where the perturbation
could be the introduction of extra supports in a beam structure, and
the first eigenvalue is proportional to the critical pressure at which
the structure will start to buckle).  As mentioned above, our result
also provides a unifying framework for the nodal counting theorems.
In this manuscript we derive and strengthen one of them as an example.
Finally, the classical tool of spectral theory, the Birman--Schwinger
operator (or Schur complement in linear algebra), arises naturally as the Hessian with respect to variation of
the perturbation.  Its eigenfunctions are interpreted as giving the
directions in which the eigenvalue changes the most.

\section{Main results in the simplified form}\label{S:results}

Let $\cH$ and $\cK$ be separable complex Hilbert spaces and denote by
$\cC:=\cC(\cH,\cK)$ the Banach space of compact linear operators from
$\cH$ to $\cK$.

Let $\lo$ be an eigenvalue of a bounded below self-adjoint operator
$S : \cH \to \cH$, lying below the essential spectrum of $S$; let $f$
be the corresponding eigenfunction.  Consider the perturbed operator
$S+K_0^*K_0$, $K_0 \in \cC$ and assume $K_0 f = 0$ so that
$\lo$ is also an eigenvalue of $S+K_0^*K_0$.  For a self-adjoint operator $A$ we denote by $N(\lo; A)$
the number of eigenvalues of $A$ below $\lo$ and denote by $\sigma$
the \term{spectral shift}
\begin{equation*}
  \sigma = N(\lo;S) - N(\lo; S+K_0^*K_0).
\end{equation*}
We will now allow the perturbation $K_0$ to vary in the most general
way, considering $H(K) = S+K^*K$ with $K$ ranging over an open
neighborhood of $K_0$ in $\cC$.

Denote by $F$ the subspace of $\cC$ consisting of the rank one
operators acting as $x\mapsto \langle f, x\rangle_\cH \psi$, where
$\psi \in \cK$.  The subspace $F$ is isometric to $\cK$ and we have
the direct\footnote{Throughout the paper we use the notation $\oplus$
  for a direct sum and the notation $\dot\oplus$ for an orthogonal sum of
  subspaces.  We note, however, that restricted to the
  subspace of $\cC$ consisting of Hilbert--Schmidt operators, the
  decomposition $F \oplus \Fo$ becomes orthogonal.} decomposition
$\cC = F \oplus \Fo$, where $\Fo$ is the subspace of operators
$K\in\cC$ vanishing on $f$, i.e. such that $Kf=0$.

Here is a
somewhat simplified version of the main result:

\begin{theorem}[Main Theorem --- a simplified form]
  \label{thm:allK_variation}
  Let $\lo$ be a simple eigenvalue of $S$ with eigenfunction $f$ and
  let $K_0f = 0$. Consider the family
  \begin{equation}
    \label{eq:HK_family}
    H(K) = S + K^* K, \qquad K \in \cC(\cH, \cK).
  \end{equation}
  Assume that the eigenvalue $\lo$ is also simple in the spectrum of $H(K_0)$ and let
  the function $\Lambda(K) := \lambda(H(K))$ be its real analytic
  continuation defined in a neighborhood of $K_0$ in $\cC$.  Then
  \begin{enumerate}
  \item $K=K_0$ is a critical point of the function $\Lambda(K)$,
  \item the Hessian of $\Lambda(K)$ at $K=K_0$ is zero on the space
    $\Fo$ and is reduced by the decomposition $\cC = F \oplus \Fo$,
  \item the Hessian restricted to $F$ is a quadratic form on $F$
    and its Morse index (number of its negative eigenvalues) is equal
    to the spectral shift $\sigma$.
  \end{enumerate}
\end{theorem}
By a ``critical point'' we mean that the $\R$-linear terms in the
analytic expansion of $\Lambda(K)$ at $K=K_0$ are zero.  By the
``Hessian'' we mean the quadratic terms of the real analytic expansion
of $\Lambda(K)$.  The theorem above directly follows from a more
general result, Theorem~\ref{thm:main_additive} in
Section~\ref{sec:proof_main}, where we drop such restrictions as the
simplicity of $\lo$ in the spectrum of $S$ and $H(K)$ being a positive
perturbation of $S$.

It is also not necessary to vary $K$ in \emph{all} possible directions
to recover the spectral shift as the Morse index of $\Lambda(K)$.
Restricting $K$ to a submanifold in $\cC$ transversal to $\Fo$ we will
obtain the same result in Theorems \ref{T:Curved} and \ref{T:Curved2}.

\section{Morse indices and Schur complements}
\label{sec:schur}

\subsection{Morse indices}

We define first the indices that are involved in our main results.  We
denote by $\sigma(H)$ the spectrum of $H$ and by $\spess(H)$ its
essential spectrum, defined as the complement of the set of
$\lambda \in \C$ such that $H-\lambda$ is Fredholm.  We recall that
for self-adjoint operators, $\sigma(H)$ is the disjoint union of
$\spess(H)$ and the discrete spectrum $\sigma_d(H)$, i.e.\ the set of
isolated eigenvalues of finite multiplicity.

\begin{definition}
  \label{def:inertia}
  Let $H$ be a self-adjoint operator on $\cH$.  For an interval
  $I\subset \R$, we denote by $E_I$ the (projector-valued) spectral
  measure of $I$ corresponding to $H$.  We
  define two indices $i_-$ and $i_0$ (which may be infinite) as
  follows:
  \begin{align}
    \label{eq:inertia_calc}
    i_-: & = \dim\Ran E_{(-\infty,0)}, \\
    i_0: & = \dim\Ker H,
  \end{align}
  where $\Ker$ denotes the kernel of the operator and $\Ran$ denotes
  the range.

  We will refer to $i_-$ as the \term{Morse index} and to $i_0$ as
  the \term{nullity} of $H$.
\end{definition}

A well-known and very useful equivalent formula for $i_-$ (often
called Glazman's lemma, see e.g. \cite[Lemma 3.1 in Supplement
1]{ShubinBerezin_schrod}) looks as follows.

\begin{lemma}\label{L:subspaceindex}
  The Morse index $i_-$ is the maximal dimension of a subspace $\cM$
  on which operator $H$ is negative, i.e. $(x,Hx)<0$ for all
  $x\in\cM$, $x\neq 0$.
\end{lemma}

%

This interpretation of the Morse index allows for a simple, general,
and surely well known proof of the classical \emph{Sylvester's law of
  inertia}:
\begin{lemma}\label{L:sylvester}
  Let $H$ be a self-adjoint operator on $\cH$ with domain $\Dom(H)$.
  If $S$ is a bounded
  invertible operator in $\cH$, then
  $S^*HS$ is self-adjoint on the natural domain
  $S^{-1}(\Dom(H))$ and 
  and
  \begin{align}
    \label{E:sylvester}
    i_-(H)&=i_-(S^*HS),\\
    \label{eq:sylvester0}
    i_0(H)&=i_0(S^*HS).
  \end{align}
\end{lemma}

\begin{proof}
  Since $(x,S^*HSx)=(Sx,HSx)$ on $S^{-1}(\Dom(H))$, the operator
  $S^{-1}$ establishes an isomorphism between subspaces in $\Dom(H)$
  and $\Dom(S^*HS)$, which preserves the negativity property (and in
  fact, the numerical range).
\end{proof}

\subsection{Schur complement; finite dimensional case}

We recall first the notion of the Schur complement in the matrix case. Let
\begin{equation}\label{E:Schur_def}
M=\left(
  \begin{array}{cc}
    A & B \\
    C & D \\
  \end{array}
\right)
\end{equation}
be a block-matrix, with the diagonal block $D$ being invertible.
\begin{definition}\label{D:Schur}
The matrix $A-BD^{-1}C$ is called the \term{Schur complement} of $D$ in $M$ (or just \term{Schur complement}, if no confusion can arise). We denote it as follows:
\begin{equation}
M/D:=A-BD^{-1}C
\end{equation}
\end{definition}
The name (introduced by E.~Haynsworth \cite{Hay_laa68}) comes from a well known J.~Schur determinant formula \cite{Schur}, which was based on a Gauss elimination procedure reducing $M$ to the form
\begin{equation}\label{E:Schur_red}
\left(
  \begin{array}{cc}
    A-BD^{-1}C & B \\
    0 & D \\
  \end{array}
\right).
\end{equation}

\subsection{Schur complement; unbounded operators}

Let operator $H$ be as in the beginning of the section, and $P_1$ be an orthogonal projector keeping
the domain $\Dom(H)$ invariant, i.e. $P_1\Dom(H)\subset \Dom(H)$.
We denote by $\cH_1$ and $\cH_2$ the ranges of projector $P_1$ and of the complementary projector $P_2:=I-P_1$ respectively.
We thus have the orthogonal decomposition
\begin{equation}
\cH=\cH_1 \dot\oplus\cH_2.
\end{equation}
Thus, operator $H$ can be represented in the block form
\begin{equation}
  H =
  \begin{pmatrix}
    A & B \\ \widetilde{B} & D
  \end{pmatrix},
\end{equation}
where all blocks are closed operators between the corresponding
spaces. Due to self-adjointness of $H$, it checks out that $A$ and $D$
are self-adjoint in the spaces $\cH_1$ and $\cH_2$ correspondingly
with the natural domains $P_i(\Dom(H))$. Also, the operator
$\widetilde{B}:\cH_1\to \cH_2$ is adjoint to $B:\cH_2\to \cH_1$.  We
thus end up with the decomposition
\begin{equation}
  \label{eq:block_form}
  H =
  \begin{pmatrix}
    A & B \\ B^* & D
  \end{pmatrix}.
\end{equation}
A thorough study of operators represented in this form can be found in
\cite{Tretter_block}.

We need to remind the reader the following notion:
\begin{definition}\label{D:M-P}
  An operator $D^+$ is said to be a \term{generalized inverse} to $D$
  if the following equality holds:
  \begin{equation}
    \label{eq:generalized_inverse}
    DD^+D=D.
  \end{equation}
  In other words, $D^+$ is a right inverse to $D$ on the range of $D$.
\end{definition}

\begin{remark}
  Different flavors of generalized inverses exist (see, for example,
  \cite[Chap.~9]{BenIsraelGreville_GenInverses}), but the above basic
  property is sufficient for our purposes.  The reader should notice
  that an operator $D^+$ satisfying \eqref{eq:generalized_inverse}
  always exists, for example defined on $\Ran(D)$, without requiring
  $D$ to be injective or surjective.  A particular choice, satisfying
  more restrictive conditions which guarantee uniqueness, is the
  \term{Moore--Penrose (pseudo-)inverse}.
\end{remark}

The following formula, proved originally for matrices, goes back at least
to a 1968 article by Haynsworth \cite{Hay_laa68}.

\begin{theorem}
  \label{thm:BirmanSchwinger}
  Let $H$ be a self-adjoint operator on $\cH$
  and let $\cH_1 \dot\oplus \cH_2 = \cH$ be the orthogonal
  decomposition described above, in particular $P_1 \Dom(H) \subset
  \Dom(H)$.
  \begin{enumerate}
  \item \label{item:Haynsworth_single} If $0\not\in\spess(D)$ and
    for some constant $C>0$ and all $x\in\cH_2$,
    \begin{equation}\label{E:B<D}
      \|Bx\|_{\cH_1}\leq  C\|Dx\|_{\cH_2},
    \end{equation}
    then, with any choice $D^+$ of the generalized inverse of $D$, the
    operator $A - B D^{+} B^*$ is self-adjoint and
    \begin{align}
      \label{eq:Haynsworth-}
      i_-(H) &= i_-(D) + i_-\left(A - B D^{+} B^*\right),\\
      \label{eq:Haynsworth0}
      i_0(H) &= i_0(D) +i_0\left(A - B D^{+} B^*\right),
    \end{align}
    assuming the relevant indices are finite.
  \item \label{item:Haysworth_double}
    If $0 \not\in \spess(A) \cup \spess(D)$ and, in addition to \eqref{E:B<D},
    \begin{equation}\label{E:B*<A}
      \|B^*x\|_{\cH_2}\leq  C\|Ax\|_{\cH_1},
    \end{equation}
    for some $C>0$, one has
    \begin{align}
      \label{E:HaynsworthA-}
      &i_-(A)-i_-(D)
      = i_-\left(A - B D^{+} B^*\right) - i_-\left(D - B^* A^{+} B\right),\\
      \label{E:HaynsworthA0}
      &i_0(A)-i_0(D)
      = i_0\left(A - B D^{+} B^*\right) - i_0\left(D - B^* A^{+} B\right),
    \end{align}
    assuming the relevant indices are finite.
  \end{enumerate}
\end{theorem}

\begin{remark}
  \label{R:schurcomp}\indent
  \begin{enumerate}
  \item Equations~\eqref{eq:Haynsworth-} and \eqref{eq:Haynsworth0}
    are known in the matrix case as the ``Haynsworth formula'',
    usually formulated under the condition of invertibility of $D$.
    Extended version using various flavors of generalized matrix
    inverses are also known, see e.g.\
    \cite{CarHayMar_siamjam74,HanFuj_laa85,Mad_laa88,Jon+_laa87,Tia_laa10}
    and \cite[Thm~A.1]{BerCanCoxMar_prep20}.  To the best of our
    knowledge, the present version might be the first one for unbounded
    operators with not necessarily invertible $D$ (however, several
    similar results are contained in \cite{Tretter_block}).  Extending
    Definition~\ref{D:Schur}, we will call the operator
    $A - B D^{+} B^*$ the \term{Schur complement} $M/D$ of $D$ in $M$.
  \item Condition (\ref{E:B<D}) implies the inclusion
    $\Ker D\subset \Ker B$. In finite dimension, they are equivalent.
  \item Part~\eqref{item:Haynsworth_single} of the theorem has a
    symmetric counterpart:
    if $0\not\in\spess(A)$ and (\ref{E:B<D}) is replaced with \eqref{E:B*<A},
     one has
     \begin{align}
       \label{eq:HaynsworthA-}
       i_-(H) &= i_-(A) + i_-\left(D - B^* A^{+} B\right),\\
       \label{eq:HaynsworthA0}
       i_0(H) &= i_0(A) +i_0\left(D - B^* A^{+} B\right),
     \end{align}
     assuming the relevant indices of $H$ and $A$ are finite.  This is
     used, in particular, to prove part~\eqref{item:Haysworth_double}
     of the theorem.
  \item Part~(\ref{item:Haysworth_double}) of the theorem shows that
    the spectral shift between the operators $A$ and $D$ is the same
    as between their Schur complements. In particular, if indices of
    $A$ and $D$ coincide, then those of $M/A$ and $M/D$ also do.
  \item According the Definition~\ref{def:inertia}, we need 0 to be
    away from the essential spectrum of the corresponding operator
    in order to have the indices $i_0$ and $i_-$ well-defined.  In
    particular, $i_0$ will be finite.  But we need not assume
    finiteness of $i_-$ in order to use \eqref{eq:Haynsworth0} or
    \eqref{E:HaynsworthA0}.
  \end{enumerate}
\end{remark}

Our proof of Theorem~\ref{thm:BirmanSchwinger} mostly adheres to the
existing proofs for matrices, except for the use of
Lemma~\ref{L:subspaceindex} instead of the original definition of the
indices. We prove the following auxiliary statement first.

\begin{lemma}
  \label{lem:BDinverse}
  Let $D$ be a self-adjoint operator on $\cH_2$ and let $0\not\in\spess(D)$.
  If condition~\eqref{E:B<D} holds for an operator $B:\cH_2 \to
  \cH_1$, then the following properties hold:
  \begin{enumerate}
  \item the operator $BD^+B^*$ does not depend on the choice of the
    generalized inverse $D^+$,
  \item \label{item:BDDB} for an arbitrary choice of $D^+$, we have
    $BD^+D = B$,
  \item \label{item:Moore} there exists a self-adjoint choice of
    $D^+$, such that the operator $BD^+$ is bounded,
  \end{enumerate}
\end{lemma}

\begin{proof}
  Since zero is not in the essential spectrum of $D$, $D$ is Fredholm
  and its range $D$ is closed.  From inequality~\eqref{E:B<D} we have
  $\Ker D \subset \Ker B$ and therefore $\Ran B^* \subset
  \overline{\Ran D} = \Ran D$.

  Let now $D^+$ be an arbitrary generalized inverse of $D$.
  Equation~\eqref{eq:generalized_inverse} implies that
  $D(D^+D x - x)=0$ for any $x\in \Dom D$, or, equivalently
  \begin{equation}
    \label{eq:remainderDD}
    D^+Dx - x \in \Ker D \subset \Ker B.
  \end{equation}
  We apply $B$ to \eqref{eq:remainderDD} and obtain
  \begin{equation}
    \label{eq:BDDB}
    BD^+Dx - Bx = 0,
  \end{equation}
  establishing part~(\ref{item:BDDB}) of the lemma.

  Since $\Ran B^* \subset \Ran D$, for a given $y$ there exists an
  $x \in \Dom D$ such that $B^*y = Dx$.  Then \eqref{eq:BDDB} becomes
  $BD^+B^*y = Bx$ and, since $x$ did not depend on the choice of
  $D^+$, neither does the operator $BD^+B$.

  Finally, let $P$ be the orthogonal projection onto the range of $D$,
  then $D$ restricted to the space $P\cH_1=\Ran D$ is self-adjoint and has a
  bounded inverse, which we denote by $D_P^{-1}$.  The generalized
  inverse\footnote{This is, in fact, the Moore--Penrose inverse.}
  $P^+ = PD_P^{-1}P = D_P^{-1} \oplus 0$ is self-adjoint (the latter
  representation is with respect to the decomposition
  $\cH_1 = \Ran D \dot\oplus \Ker D$).  Furthermore, \eqref{E:B<D} yields
  \begin{equation*}
    \| B D^+ x \|_{\cH_2} \leq C \| D D^+ x \|_{\cH_1}
    = C \|P x \|_{\cH_1} \leq C \|x\|_{\cH_1},
  \end{equation*}
  establishing part~(\ref{item:Moore}).
\end{proof}

\begin{proof}[Proof of Theorem~\ref{thm:BirmanSchwinger}]
  According to the lemma, it is enough to prove
  \eqref{eq:Haynsworth-}-\eqref{eq:Haynsworth0} for one particular
  choice of $D^+$ and we will use the self adjoint $D^+$ such that the
  operator $BD^+$ is bounded. This implies boundedness and
  invertibility of the operator matrix
  \begin{equation*}
    Q :=
    \begin{pmatrix}
      I & B D^{+} \\
      0 & I
    \end{pmatrix}.
  \end{equation*}

  We can now represent the operator matrix $H$ as follows:
  \begin{equation}
    \label{eq:Schur_factorization}
    H =
    \begin{pmatrix}
      A & B \\ B^* & D
    \end{pmatrix}
    =
    Q
    \begin{pmatrix}
      A - B D^{+} B^* & 0 \\
      0 & D
    \end{pmatrix}
    Q^*.
  \end{equation}
  Indeed, direct
  calculation shows
  \begin{equation*}
    Q
    \begin{pmatrix}
      A - B D^{+} B^* & 0 \\
      0 & D
    \end{pmatrix}
    Q^*
    =
    \begin{pmatrix}
      A - B D^{+} B^* + B D^+ D (D^+)^*B^* & B D^+ D \\
      D(D^+)^* B^* & D
    \end{pmatrix},
  \end{equation*}
  and the identities $B D^+ D = B$ and $(D^+)^* = D^+$ do the rest.

  From Sylvester's law of inertia (Lemma~\ref{L:sylvester}), we
  have that $A - B D^{+} B^*$ is self-adjoint and
  \begin{equation}
    \label{eq:indexSchur}
    i_-\left(H\right)
    =
    i_-\left(
    \begin{pmatrix}
      A - B D^{+} B^* & 0 \\
      0 & D
    \end{pmatrix}
    \right) = i_-(D) + i_-\left(A - B D^{+} B^*\right),
  \end{equation}
  by definition~\ref{def:inertia} and the orthogonal decomposition of
  the spectral projectors of the block-diagonal operator matrix.  The equality
  for $i_0$ is established in the same way.

  To establish the second part of the theorem, we reverse the roles
  of $A$ and $D$ and (with estimate~\eqref{E:B*<A} playing the role of
  \eqref{E:B<D}) obtain
  \begin{align*}
    i_-(H) &= i_-(A) + i_-\left(D - B^* A^{+} B\right),\\
    i_0(H) &= i_0(A) +i_0\left(D - B^* A^{+} B\right),
  \end{align*}
  Using \eqref{eq:Haynsworth-} and \eqref{eq:Haynsworth0} to eliminate
  $i_-(H)$ and $i_0(H)$ we obtain the desired result.
\end{proof}

\begin{remark}
  The Schur complement technique (and its close relatives) is very
  natural and thus has been re-invented many times under various
  guises, e.g.\ as Dirichlet-to-Neumann operators, $m$-functions for
  ODEs, Birman--Schwinger approach, and probably many others.
\end{remark}

\section{The main result}\label{sec:proof_main}

Let $\cH$ and $\cK$ be separable complex Hilbert spaces and, as before, we denote by
$\cC(\cH,\cK)$ the Banach space of compact operators from $\cH$ to
$\cK$.

\begin{definition}\label{D:F}
  We denote by $F$ the subspace of $\cC(\cH,\cK)$ consisting of the
  operators $K_\psi$ acting as
  \begin{equation}
    \label{E:K_chi}
    K_\psi: x\in \cH \mapsto \langle f,x\rangle\psi,
  \end{equation}
  for some $\psi \in \cK$.

  The subspace $\Fo$ consists of operators $K$ such that $Kf=0$.
\end{definition}

\begin{remark}
  Alternatively, $F$ can be defined as the subspace of
  $K \in \cC(\cH,\cK)$ such that
  $\Ker K \supset f^\perp := \{u\in\cH \colon \langle f, u \rangle =
  0\}$.
\end{remark}

\begin{lemma}\label{L:FF}\indent
  \begin{enumerate}
  \item The correspondence $\psi\leftrightarrow K_\psi$ is an isometry
    between $\cK$ and $F$.
  \item  $F\oplus \Fo=\cC$.
  \end{enumerate}
\end{lemma}

\begin{proof}
  To compute the operator norm of $K_\psi$ we use Cauchy--Schwartz
  inequality, keeping in mind that $\|f\|=1$,
  \begin{equation*}
    \| K_\psi x \|_\cK = |\langle f,x\rangle | \|\psi\|_\cK
    \leq \|\psi\|_\cK \|x\|_\cH,
  \end{equation*}
  with equality achieved when $x = f$.

  The splitting of a $K \in \cC(\cH,\cK)$ is given explicitly by
  \begin{equation*}
    \langle f, \cdot \rangle Kf \in F
    \quad\mbox{and}\quad
    K - \langle f, \cdot \rangle Kf \in \Fo.
  \end{equation*}
\end{proof}
Let $H_0$ be a bounded below self-adjoint operator on $\cH$ and $\lo$
be its simple isolated eigenvalue with the corresponding normalized
eigenfunction $f$. Assume that the spectrum of $H_0$ below $\lo$
consists of finitely many eigenvalues of finite multiplicity.
Suppose also that $H_0$ is of the form
\begin{equation}
  \label{eq:T_representation}
  H_0 = S + K_0^* \Omega K_0,
  \qquad \mbox{with} \quad K_0 \in \Fo,\mbox{ i.e. } K_0f = 0,
\end{equation}
where $\Omega$ is a bounded invertible self-adjoint
operator\footnote{A simple and essentially sufficient example is when
  $\Omega = (-\Id_{\cK_-}) \oplus \Id_{\cK_+}$ with respect to some
  orthogonal decomposition $\cK_- \dot\oplus \cK_+ = \cK$, with
  $i_-(\Omega)= \dim(\cK_-) < \infty.$} on $\cK$, whose spectrum below
zero consists of finitely many eigenvalues of finite multiplicity, so
$i_0(\Omega)=0$ and $i_-(\Omega)<\infty$.

Since $K_0f=0$, $f$ is also an eigenfunction
of $S$ with the same eigenvalue $\lo$.  The essential spectrum of $S$
also lies above $\lo$, although $\lo$ may no longer be simple in the
spectrum of $S$.

Let, as before, $i_-(H_0-\lo)$ be the number of eigenvalues (counted
with multiplicity) of $H_0$ below $\lo$ and denote by $\sigma$ the
\term{spectral shift}
\begin{equation}
  \label{eq:spectral_shift}
  \sigma: = \sigma(\lo; S, H_0):= i_-(S-\lo) - i_-(H_0-\lo).
\end{equation}
\begin{remark}
Notice that when $\Omega$ is positive, the spectral shift is also positive.
\end{remark}

Consider the family of operators
\begin{equation}
  \label{eq:TK_family}
  H(K) = S + K^* \Omega K, \qquad K \in \cC(\cH,\cK),
\end{equation}
so, in particular, $H(K_0)=H_0$.

Since $\lo$ is a simple eigenvalue of $H_0 = H(K_0)$, there is a
real-analytic branch $\Lambda(K)$ of the eigenvalues of $H(K)$ that
is the continuation of $\lo$ defined to a neighborhood $\Pi$ of
$K_0$ in $\cC(\cH,\cK)$.  Real analyticity means, in particular, the existence
of the expansion
\begin{equation}
  \label{eq:analytic_expansion}
  \Lambda(K) = \lambda_0 + A_1(\dK) + A_2(\dK) +
  O(\|\dK\|^3),
\end{equation}
where $\dK := K - K_0$ and $A_m : \cC(\cH,\cK) \to \R$ is
homogeneous of degree $m$,
\begin{equation*}
  A_m(\alpha \, \dK) = \alpha^m A_m(\dK), \qquad \alpha \in \R.
\end{equation*}
If $A_1 \equiv 0$, we say that $K_0$ is a \term{critical point} of
$\Lambda(K)$; the quadratic term $A_2$ will be called the
\term{Hessian} of $\Lambda(K)$ at $K_0$.

\begin{theorem}[Main Theorem --- General Form]
  \label{thm:main_additive}
  ~
  \begin{enumerate}
  \item \label{item:CP} The function $\Lambda(K)$ has a critical point
    at $K=K_0$.
  \item \label{item:Hess_reduced} The Hessian $A_2$ of $\Lambda(K)$
    at $K_0$ is zero on the space $\Fo$ and is reduced by the
    decomposition $\cC(\cH,\cK)=F\oplus \Fo$ in the following
    sense: for any $\dK_\psi \in F$ and $\dK_a \in \Fo$,
    \begin{equation}
      \label{eq:reduced}
      A_2(\dK_\psi+\dK_a) = A_2(\dK_\psi).
    \end{equation}
    Restricted to $F$ (which is viewed as a Hilbert space isometric to
    $\cK$), the Hessian $A_2$ is a quadratic form.
  \item \label{item:Hess_index} The Morse index (cf.\
    Lemma~\ref{L:subspaceindex}) of the Hessian $A_2$ on $F$ is
    \begin{equation}
      \label{eq:morse_index}
      i_-\left(A_2 |_F\right) = \sigma + i_-(\Omega),
    \end{equation}
    where $\sigma$ is the spectral shift defined in
    \eqref{eq:spectral_shift}. In particular, if $\Omega$ is positive,
    then the Hessian's Morse index is equal to the spectral shift.
  \item \label{item:Hess_zero} The nullity of the Hessian $A_2$ on $F$
    is
    \begin{equation}
      \label{eq:nullity}
      i_0\left(A_2|_F\right) = m-1,
    \end{equation}
    where $m$ the multiplicity of the eigenvalue $\lo$ in the spectrum
    of $S$.  In particular, if $\lo$ is a simple eigenvalue of $S$,
    the critical point $K=K_0$ is non-degenerate with respect to
    variations $\dK \in F$.
  \item \label{item:operator} The quadratic form $A_2|_F$ corresponds
    to the bounded self-adjoint operator on $\cK$,
    \begin{equation}
      \label{eq:Hessian_operator0}
      Q := \Omega - \Omega K_0 \big(H_0-\lo\big)^{+} K_0^* \Omega,
    \end{equation}
    which is a compact perturbation of the operator $\Omega$.
  \end{enumerate}
\end{theorem}

\begin{remark}
  \label{rem:BirmanSchwinger}
  The operator $Q$ of \eqref{eq:Hessian_operator0} often arises in
  spectral analysis of perturbations of the form \eqref{eq:TK_family}
  (see \cite{KonKur_jfsut66,How_arma70,Yaf_scattering_v1}); it is an
  operator-valued Herglotz function \cite{GesKalMakTse_otaa01} which
  is well-known for its role in Birman--Schwinger principle and
  spectral shift, see
  \cite{GesMakNab_otaa99,Pus_ahp09,BehGesNak_ma18,BehElsGes_prep20}
  and references therein.  It is the Birman--Schwinger principle (see, e.g.
  \cite[Thm.~4.1]{Pus_ahp09}) that
  extracts parts \eqref{item:Hess_index} and \eqref{item:Hess_zero} of
  our Theorem from part~\eqref{item:operator}.  We keep
  our proof self-contained by relating everything to Schur complement
  and Theorem~\ref{thm:BirmanSchwinger}.  The link between Schur
  complement and Birman--Schwinger operator has also been observed
  before \cite{Tretter_block}.
\end{remark}

\begin{remark}
  The statement of the theorem may seem puzzling at first: how could
  any information about the operator $S$ be extracted from small
  perturbations of the ``far away'' operator $H_0$? This confusion is
  resolved by realizing that the operator $K_0$, whose small
  perturbations are used, is known, and thus $S$ is defined by $H_0$
  and $K_0$.
\end{remark}

\begin{remark}
  The spectral shift $\sigma$ defined by \eqref{eq:spectral_shift} can
  be negative, but it cannot exceed the rank of the negative part of
  the perturbation.  Thus $\sigma + i_-(\Omega) \geq 0$, which we
  would expect for a Morse index.
\end{remark}

\begin{proof}[Proof of Theorem~\ref{thm:main_additive}]
  Let $K$ be close to $K_0$ and $z$ be in a punctured neighborhood of
  $\lo$. The condition of $z$ being in the spectrum of $H(K)$ is
  equivalent to
  \begin{equation}
    1 = i_0\big( H(K)-z \big)
    = i_0\big( (S-z)-K^*(-\Omega)K \big).
  \end{equation}
  Consider the block operator on $\cH \oplus \cK$
  \begin{equation*}
    \begin{pmatrix}
      A & B \\ B^* & D
    \end{pmatrix}
    :=
    \begin{pmatrix}
      S-z & K^* \\ K & -\Omega^{-1}
    \end{pmatrix},
  \end{equation*}
  which is self-adjoint as a bounded perturbation of a self-adjoint
  block-diagonal operator.  The blocks $S-z$ and $-\Omega^{-1}$ are
  invertible and therefore $i_0(A)=i_0(D)=0$.  Using identity
  (\ref{E:HaynsworthA0}) of Theorem \ref{thm:BirmanSchwinger} we
  get\footnote{The fact that $i_-(-\Omega)$ is possibly infinite is of
    no concern since we are dealing with nullity only.} an equivalent
  condition for $z$ being equal to $\Lambda(K)$:
  \begin{equation}
    \label{eq:eig_cond}
    i_0\big( -\Omega^{-1}-K(S-z)^{-1}K^* \big)=1.
  \end{equation}
  We decompose $K$ in accordance to the direct sum $\cC(\cH,\cK) =
  F\oplus \Fo$, see Lemma~\ref{L:FF},
  \begin{equation}
    \label{eq:Kdecomp}
    K = K_\psi + K_a,
    \quad
    K_af = 0,
    \quad K_\psi = \langle f, \cdot \rangle \psi,
    \quad \mbox{with }\psi = Kf.
  \end{equation}
  The operators $K_a$ and $K_\psi$ are perturbations ``along'' $f$
  and ``lateral'' to it, correspondingly.  The operator in
  equation~\eqref{eq:eig_cond} can now be expanded as
  \begin{align*}
    \Omega^{-1} &+ (K_\psi + K_a) (S-z)^{-1} (K_a + K_\psi)^* \\
    &= \Omega^{-1} + K_a (S-z)^{-1} K_a^*
      + K_a (S-z)^{-1} K_\psi^* + K_\psi (S-z)^{-1} K_a^*
      + K_\psi (S-z)^{-1} K_\psi^* \\
    &= \Omega^{-1} + K_a (S-z)^{-1} K_a^*
      + \frac{1}{\lo-z} K_\psi K_\psi^*,
  \end{align*}
  where we used
  \begin{equation*}
    K_\psi^* = \langle \psi, \cdot \rangle f,
    \qquad
    (S-z)^{-1} K_\psi^* = \frac1{\lo-z} K_\psi^*,
    \qquad
    K_a K_\psi^* = 0,
  \end{equation*}
  to eliminate middle terms.  Furthermore, we can represent
  \begin{equation*}
    \frac1{\lo-z} K_\psi K_\psi^*
    = \frac1{\lo-z} \langle \psi, \cdot \rangle \psi
    = M_\psi \frac1{\lo-z} M_\psi^*,
  \end{equation*}
  where $M_\psi$ is the operator from $\mathbb{C}^1$ to $\cK$ acting as
  multiplication by $\psi$ and
  $M_\psi^* = \langle \psi, \cdot \rangle_\cK \colon \cK \to
  \mathbb{C}^1$ is its adjoint.

  We continue equation~\eqref{eq:eig_cond} with
  \begin{align}
    1
    &= i_0\big( -\Omega^{-1} - K(S-z)^{-1}K^* \big)
      \nonumber \\
    &= i_0\big( -\Omega^{-1} - K_a (S-z)^{-1} K_a^*
      - M_\psi \frac1{\lo-z} M_\psi^* \big)
      \nonumber \\
    \label{eq:i-0-almost}
    &= i_0\left(\lo - z
      + M_\psi^* \left(\Omega^{-1} + K_a(S-z)^{-1}K_a^*\right)^{-1}
      M_\psi \right),
  \end{align}
  where we used (\ref{E:HaynsworthA0}) on the bounded block operator on
  $\mathbb{C}^1\oplus\cK$ defined by
  \begin{equation*}
    \begin{pmatrix}
      A & B \\ B^* & D
    \end{pmatrix}
    :=
    \begin{pmatrix}
      \lo-z & M_\psi^* \\ M_\psi & -\Omega^{-1} - K(S-z)^{-1}K^*
    \end{pmatrix}.
  \end{equation*}
  The correction terms on the left-hand side of \eqref{E:HaynsworthA0}
  are zero because, for $z$ in a punctured neighborhood of $\lo$, the
  blocks $A$ and $D$ are invertible; the latter is due to the
  following simple lemma (see also \cite[Eqs.\
  (3.18)-(3.19)]{GesKalMakTse_otaa01}).

  \begin{lemma}
    \label{L:switch}
    For $z$ in a punctured neighborhood of $\lo$,
    \begin{equation}
      \label{E:switch}
      \left(\Omega^{-1} + K_a (S-z)^{-1} K_a^*\right)^{-1}
      = \Omega - \Omega K_a \big(H(K_a)-z\big)^{-1} K_a^* \Omega
    \end{equation}
  \end{lemma}
  \begin{proof}[Proof of the Lemma]
    First we observe that since $K_a f = 0$, and $K_a-K_0$ is small,
    $\lo$ is an isolated eigenvalue of $H(K_a)$.  Therefore, $z$ is in
    the resolvent set of both $S$ and $H(K_a)$.  We can now use the
    second resolvent identity for the operators $S$ and
    $H(K_a) = S+K_a^*\Omega K_a$ to directly verify that the product,
    in any order, of
    \begin{equation*}
      \Omega^{-1} + K_a (S-z)^{-1} K_a^*
      \qquad\mbox{and}\qquad
      \Omega - \Omega K_a \big(H(K_a)-z\big)^{-1} K_a^* \Omega,
    \end{equation*}
    is equal to $I_\cK$.
  \end{proof}

  We apply Lemma~\ref{L:switch} to equation \eqref{eq:i-0-almost} to get
  \begin{equation*}
    i_0\left(\lo - z + M_\psi^* \big(\Omega - \Omega K_a
      \big(H(K_a)-z\big)^{+} K_a^* \Omega\big) M_\psi \right) = 1.
  \end{equation*}
  Obviously, the generalized inverse $\big(H(K_a)-z\big)^{+}$ coincides with
  the inverse of $H(K_a)-z$ in a punctured neighborhood of
  $\lo$.  However, because $\Ran(K_a^*)$ is orthogonal to
  $\Ker\big(H(K_a)-\lo\big)$, the expression
  $K_a \big(H(K_a)-z\big)^{+} K_a^*$ is now well-defined and
  \emph{continuous} in $z$ up to and including the point $z=\lo$.

  Finally, we use the definition of $M_\psi$ and observe that the
  argument of $i_0$ is a scalar, resulting in the scalar equation for
  $z$ to be the eigenvalue of $H(K)$, i.e. the value of
  $\Lambda(K) = \Lambda(K_a + K_\psi)$,
  \begin{equation}
    \label{eq:eigenvalue_condition}
    z = \lo + \left\langle \psi, \left(
        \Omega - \Omega K_a \big(H(K_a)-z\big)^{+} K_a^* \Omega
      \right) \psi \right\rangle.
  \end{equation}

  We now use the analyticity of $\Lambda(K)$ to estimate the relevant
  terms with respect to the perturbation $\dK = K-K_0$,
  \begin{align*}
    &z = \Lambda(K) = \lo + O(\|\dK\|),\\
    &\psi = Kf = (K-K_0)f = O(\|\dK\|),\\
    &K_a = K_0 + O(\|\dK\|).
  \end{align*}
  Keeping only the leading order of the scalar product in
  \eqref{eq:eigenvalue_condition} results in
  \begin{equation}
    \label{eq:eigenvalue_quadratic_expansion}
    \Lambda(K) = \lo + \left\langle \psi, \left(
        \Omega - \Omega K_0 \big(H_0-\lo\big)^{+} K_0^* \Omega
      \right) \psi \right\rangle + O(\|\dK\|^3).
  \end{equation}

  Comparing with expansion~\eqref{eq:analytic_expansion} we
  immediately identify
  \begin{align}
    \label{eq:gradient}
    &A_1(\dK) \equiv 0,\\
    \label{eq:Hessian}
    &A_2(\dK) = \left\langle \dK f, \left(
      \Omega - \Omega K_0 \big(H_0-\lo\big)^{+} K_0^* \Omega
      \right) \dK f \right\rangle.
  \end{align}
  Since $\dK f = \dK_\psi f = \psi$, the Hessian $A_2$ does not depend
  on the part of the perturbation from $\Fo$, completing the proof of
  parts (\ref{item:CP}) and (\ref{item:Hess_reduced}) of the theorem.
  The Hessian $A_2$ restricted to $F$ identified with $\cK$ (see
  Lemma~\ref{L:FF}) corresponds to the self-adjoint operator $Q
  \colon \cK \to \cK$,
  \begin{equation}
    \label{eq:Hessian_operator}
    Q := \Omega - \Omega K_0 \big(H_0-\lo\big)^{+} K_0^* \Omega,
  \end{equation}
  which is a compact perturbation of the bounded operator $\Omega$,
  establishing part~(\ref{item:operator}) of the Theorem.

  Aiming to use Theorem \ref{thm:BirmanSchwinger} again, we let
  \begin{equation*}
    \begin{pmatrix}
      A & B \\ B^* & D
    \end{pmatrix}
    :=
    \begin{pmatrix}
      \Omega & \Omega K_0 \\
      K_0^*\Omega & H_0-\lo
    \end{pmatrix},
  \end{equation*}
  which is self-adjoint as a bounded perturbation of a block-diagonal operator.
  We compute
  \begin{align*}
    D-B^*A^+B&= H_0 -\lo - K_0^*\Omega K_0 = S-\lo,\\
    A-BD^+B^*&=\Omega - \Omega K_0(H_0-\lo)^{+} K_0^*\Omega = Q.
  \end{align*}
  Since the conditions of part (\ref{item:Haysworth_double}) of
  Theorem \ref{thm:BirmanSchwinger} are clearly satisfied, we can use
  equations~\eqref{E:HaynsworthA-} and \eqref{E:HaynsworthA0} to get
  \begin{equation}
    i_-(\Omega)-i_-(H_0-\lo)
    =i_-(Q)-i_-(S-\lo)
  \end{equation}
  and
  \begin{equation}
    i_0(\Omega)-i_0(H_0-\lo)
    =i_0(Q)-i_0(S-\lo).
  \end{equation}
  Taking into account notation $\sigma=i_-(S-\lo)-i_-(H_0-\lo)$ and
  $m=i_0(S-\lo)$, as well as the identities $i_0(\Omega)=0$ and
  $i_0(H_0-\lo)=1$, we get statements (\ref{item:Hess_index}) and
  (\ref{item:Hess_zero}) of the theorem.
\end{proof}

\subsection{Restricted variation}
\label{sec:curved}

It is also possible to restrict variations $K$ of $K_0$ to live on a
submanifold of $\cL$.  The next results specify how much
freedom of variation is enough to capture the right Morse index.

Assume, as before, that $\lo$ is simple eigenvalue of $S+K_0^*K_0$ and
an eigenvalue of $S$ with the same eigenfunction $f$ (the latter may
have a multiplicity $m$). Let also subspaces $F,\Fo\subset \cC$ be
defined as before. We will also denote by $\Pi$ the projector onto $F$
parallel to $\Fo$. After identifying $F$ with $\cK$, this mapping
becomes very simple: $K\mapsto Kf$.

Let $\cN\subset \cC$ be a real $C^2$-smooth Banach sub-manifold, such
that $K_0\in \cN$, and let $T_{K_0}\cN\subset \cC$ denote the tangent
space to $\cN$ at $K_0$.

We will be interested in the perturbations of the following form:
\begin{equation}
\Lambda(K):=K\in\cN \mapsto \Lambda(S+K^*\Omega K).
\end{equation}
In particular, $\Lambda(K_0)=\lo$.  Now the following version of the
main result holds:

\begin{theorem}\label{T:Curved}
  Suppose that $\Pi: T_{K_0}\cN \to F$ is an isomorphism (which
  gives $T_{K_0}\cN$ the structure of a Hilbert space). Then
\begin{enumerate}
  \item The point $K_0$ is a critical
    point of the function $\Lambda: K\in\cN\mapsto
    \Lambda(S+K^*\Omega K)$;
  \item The Hessian of $\Lambda$ at $K_0$ is a quadratic form on
    $T_{K_0}\cN$ whose Morse index is equal to $\sigma + i_-(\Omega)$
    and whose nullity is $m-1$, where $\sigma$ is the spectral shift
    and $m = i_0(S-\lo)$.
  \end{enumerate}
\end{theorem}

\begin{proof}
  The Hessian of $\Lambda$ on $\cN$ is the restriction of Hessian on
  $F \oplus \Fo$ to $T_{K_0}$.  For any $K\in T_{K_0}\cN$ we have
  $A_2(K) = A_2(\Pi K)$ by
  Theorem~\ref{thm:main_additive}(\ref{item:Hess_reduced}).  The rest
  follows from Lemma~\ref{L:sylvester} (with $S=\Pi$) and the results
  of Theorem~\ref{thm:main_additive}.
\end{proof}

It is straightforward to upgrade this theorem to the following less restrictive statement:

\begin{theorem}\label{T:Curved2}
  Suppose that $\Pi: T_{K_0}\cN \to F$ is surjective (i.e., $\cN$
  is transversal to $\Fo$ at their common point $K_0$). Then
\begin{enumerate}
  \item The point $K_0$ is a critical
    point of the function $\Lambda: K\in\cN\mapsto
    \Lambda(S+K^*\Omega K)$;
  \item The Hessian of $\Lambda$ at $K_0$ (which is a function on
    $T_{K_0}\cN$) pushes down to a quadratic form
    on the space $T_{K_0}\cN/(T_{K_0}\cN\bigcap \Fo)$.  The latter
    space is given Hilbert space structure by $\Pi$.
  \item the Morse index of this quadratic form is equal to
    $\sigma+i_-(\Omega)$ and its nullity is $m-1$.
  \end{enumerate}
\end{theorem}

\section{Examples and applications}

\subsection{A numerical example}
\label{sec:num_example}

We illustrate our results with a simple numerical example.  Consider
the $4\times4$ matrix family
\begin{equation}
  \label{eq:H_num}
  H(t, \vec{s}) =
  \begin{pmatrix}
    0 & 0 & 0 & 0 \\
    0 & 1 & 0 & 0 \\
    0 & 0 & -1 & 0 \\
    0 & 0 & 0 & -2
  \end{pmatrix}
  + t K(\vec{s})^* K(\vec{s}),
  \qquad t\in \R,\ \vec{s}\in\C^2,
\end{equation}
where
\begin{equation}
  \label{eq:K_num}
  K(\vec{s}) = K_0 + s_1 K_1 + s_2 K_2,
  \quad
  K_0 =
  \begin{pmatrix}
    0 & 0.5 & 0.5 & 1.5 \\
    0 & 1 & 2 & 1
  \end{pmatrix}.
\end{equation}
The choice of $K_1$ and $K_2$ is \emph{random}; the transversality
condition of section~\ref{sec:curved} is satisfied with probability
1.

\begin{figure}
  \centering
  \includegraphics[scale=0.5]{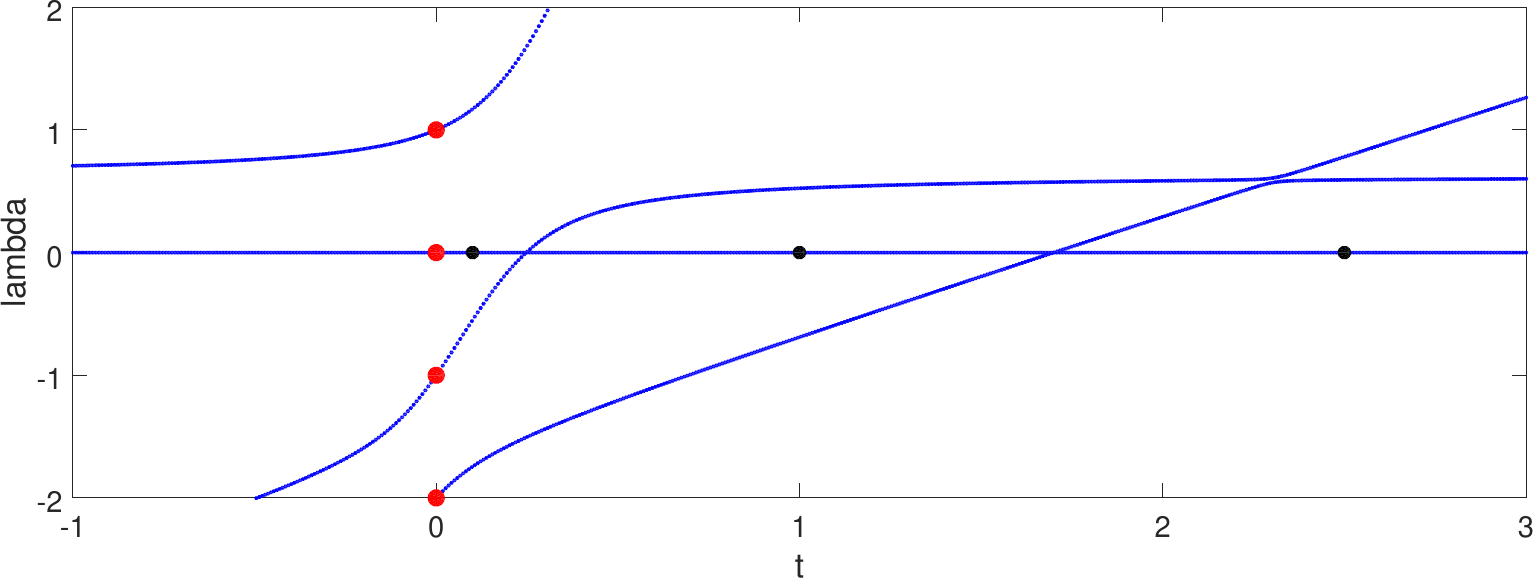}\\
  \includegraphics[scale=0.2]{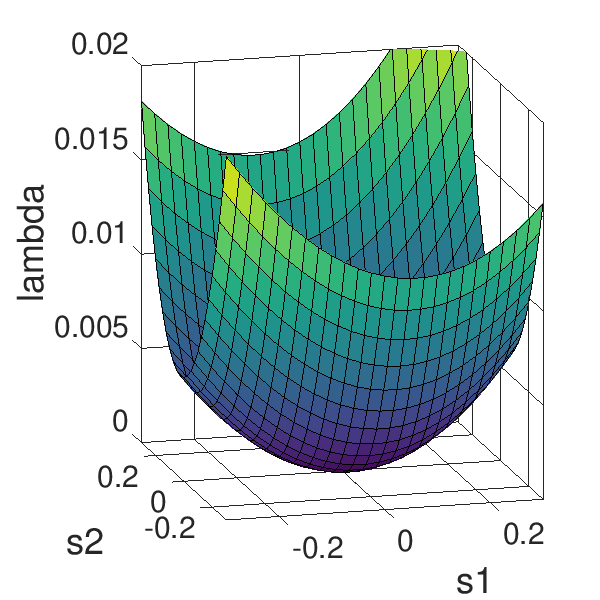}
  \includegraphics[scale=0.2]{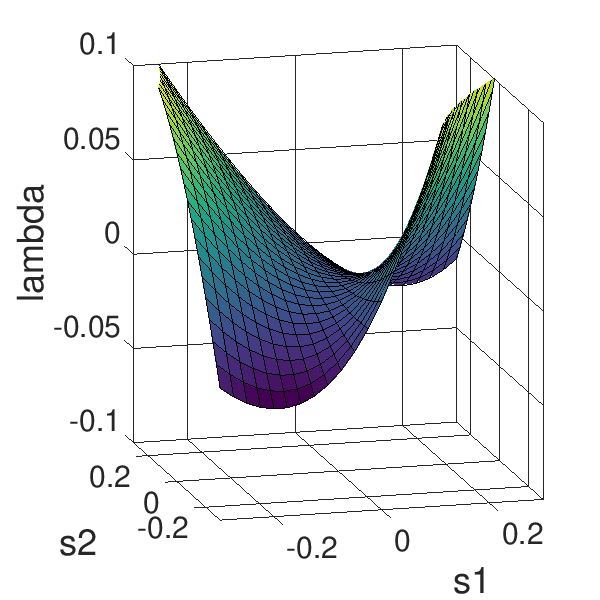}
  \includegraphics[scale=0.2]{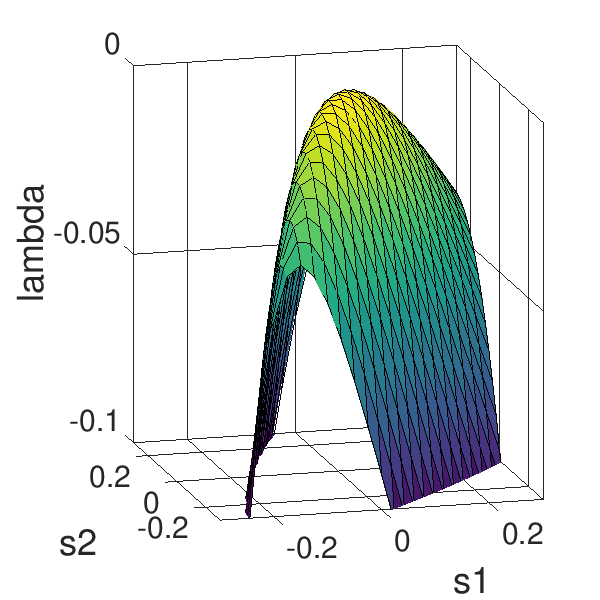}
  \caption{Top: the eigenvalues of $H(t,\vec0)$ as functions of $t$; red
    (larger) dots highlight the eigenvalues of $H(0,\vec0)$, black
    (smaller) dots indicate the values of $t$ where the lateral
    variation is explored in the bottom figures.  Bottom: the
    continuation of the eigenvalue 0 in the spectrum of
    $H(0.1,\vec{s})$, $H(1,\vec{s})$ and $H(2.5,\vec{s})$,
    correspondingly, shown as functions of $\vec{s}=(s_1,s_2)$.}
  \label{fig:spec_flow}
\end{figure}

The one-parameter family $H(t,\vec0)$ is a perturbation of $H(0,\vec0)$
along the eigenvector $e_1$ of the eigenvalue $0$.  As $t$ increases,
the eigenvalue $0$ remains constant while the other eigenvalues
increase, see Fig.~\ref{fig:spec_flow}(top).  This type of figure is
usually called \term{spectral flow}.

The spectral shift at $\lambda=0$ between $H(0,\vec0)$ and
$H(t,\vec0)$ is visualized as the number of eigenvalues crossing
$\lambda=0$ between $0$ and $t$.  Thus, at the values of $t = 0.1$,
$1.0$ and $2.5$, highlighted by black dots in
Fig.~\ref{fig:spec_flow}(top), the spectral shift is 0, 1 and 2
correspondingly.  The spectrum of the lateral variations at these
points (more precisely, the continuations of the eigenvalue 0 in the
spectrum of $H(0.1,\vec{s})$, $H(1,\vec{s})$ and $H(2.5,\vec{s})$) is
shown in the bottom row of Fig.~\ref{fig:spec_flow}.  As predicted by
Theorem~\ref{T:Curved}, the point $\vec{s} = (0,0)$ is a minimum,
saddle point and maximum correspondingly.

\subsection{An application: magnetic--nodal theorem}
\label{sec:magnetic}

We will show that a recent theorem of Berkolaiko and Colin de
Verdi\`ere, which already has two different but complicated proofs
\cite{Ber_apde13,Col_apde13}, is a simple consequence of the results
of this paper.  We start with a simple example.

\begin{example}
  \label{ex:lasso}
  Consider the matrix
  \begin{equation*}
    H(\alpha) =
    \begin{pmatrix}
      q_1 & -1 & 0 & 0\\
      -1 &q_2 & -1 & -1\\
      0 & -1 & q_3 & -e^{i\alpha}\\
      0 & -1 & -e^{-i\alpha} & q_4
    \end{pmatrix},
  \end{equation*}
  which is a matrix representation of the \term{magnetic Schr\"odinger
    operator} on the graph in Fig.~\ref{fig:lasso}, top left (precise
  definition will be given below).  We are interested in the number of
  \term{sign flips} of the $n$-th eigenvector $f$ of $H(0)$, which
  in this case can be described as the number of pairs
  $(j,k) \in \{(1,2), (2,3), (2,4), (3,4)\}$ such that $f_jf_k < 0$.
  We denote this number by $\phi_n$.

  It was discovered in \cite{Ber_apde13} that $\phi_n$ is closely
  related to local behavior of eigenvalues of $H(\alpha)$, shown in
  Fig.~\ref{fig:lasso}, right.  Whether the eigenvalue
  $\lambda_n(H(\alpha))$ experiences a minimum or a maximum at
  $\alpha=0$ is determined by whether the quantity
  $\sigma_n := \phi_n - n + 1$ is $0$ or $1$ (a part of the result is
  that $\sigma_n$ can only be $0$ or $1$ in this case).  In other
  words, $\phi_n - n + 1$ is the Morse index of
  $\lambda_n(H(\alpha))$.

  \begin{figure}
    \includegraphics{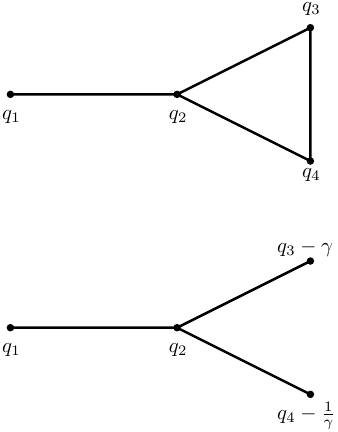}
    \hspace{0.5cm}
    \includegraphics[scale=0.5]{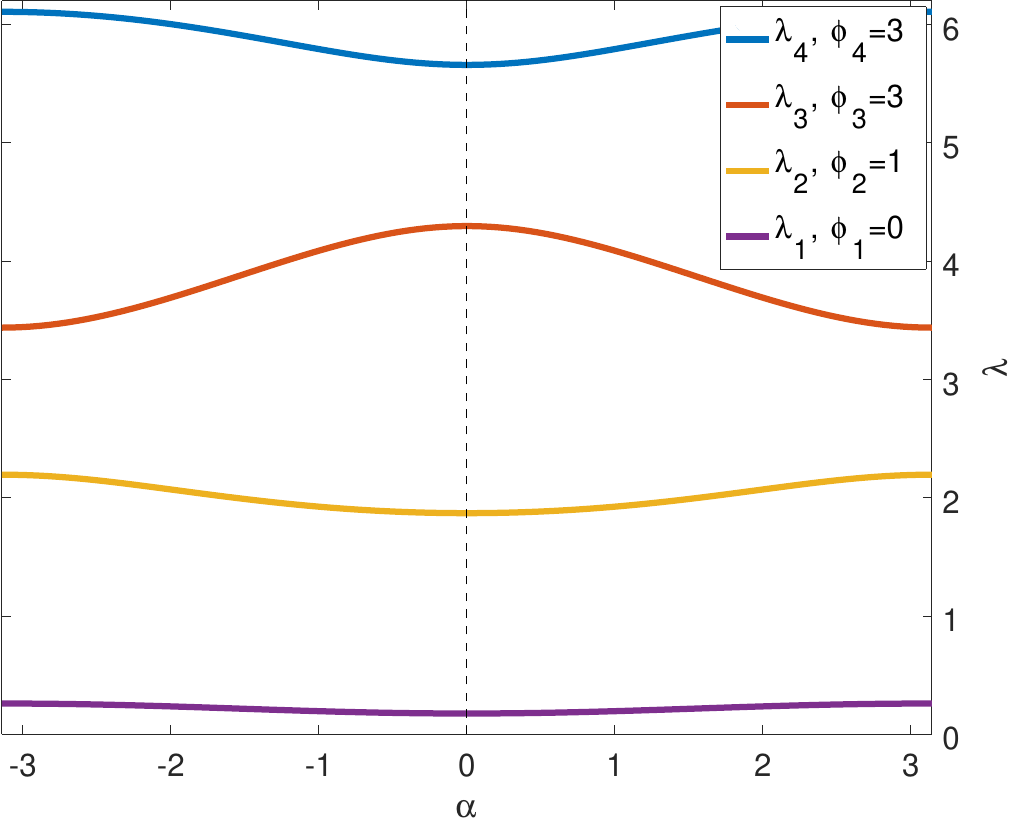}
    \caption{Top left: the graph corresponding to operator $H$ in
      Example~\ref{ex:lasso}.  Bottom left: the graph corresponding to
      operator $S$.  Right: eigenvalues of the matrix $H(\alpha)$ as
      function of $\alpha \in [-\pi, \pi]$; the legend lists the
      ``sign flip counts'' for the corresponding eigenfunction of
      $H(0)$.  We used $q_1=1$, $q_2=2$, $q_3=4$ and $q_4=5$.}
    \label{fig:lasso}
  \end{figure}

  The relation to previous results comes from the fact that
  $H(\alpha)$ can be represented as
  \begin{equation*}
    H(\alpha) =
    \begin{pmatrix}
      q_1 & -1 & 0 & 0\\
      -1 &q_2 & -1 & -1\\
      0 & -1 & q_3-\gamma & 0 \\
      0 & -1 & 0 & q_4-1/\gamma
    \end{pmatrix}
    +
    \begin{pmatrix}
      0 & 0 & 0 & 0\\
      0 & 0 & 0 & 0\\
      0 & 0 & \gamma & -e^{i\alpha}\\
      0 & 0 & -e^{-i\alpha} & 1/\gamma
    \end{pmatrix}
    =: S + P(\alpha),
  \end{equation*}
  where $\gamma$ is adjusted so that $P(0)f = 0$ for a given
  eigenfunction $f$.  The matrix $S$ is a Schr\"odinger operator on
  the tree shown in Fig.~\ref{fig:lasso}, bottom left.  It was
  established by Fiedler \cite{Fie_cmj75} that any tree satisfies
  Sturm nodal theorem: the $n$-th eigenfunction has $n-1$ sign flips.
  The spectral shift of $H$ with respect to $S$ can then be
  interpreted as ``extra number of sign flips''\footnote{Under some
    simplifying assumptions, in the quantity $\sigma = \phi - (n-1)$,
    the number of sign flips $\phi$ remains the same --- since the
    eigenfunction $f$ is unchanged --- but the position $n$ of the
    eigenvalue in the spectrum may change due to the spectral shift},
  compared to the baseline number $n-1$.  On the other hand, the
  spectral shift is equal to the Morse index of $\lambda_n(\alpha)$ by
  Theorem~\ref{thm:allK_variation} (or
  Theorem~\ref{thm:main_additive}).
\end{example}

Let us now extend and formalize the above example.  Let $H$ be a real
symmetric $N\times N$ matrix representing the \term{Schr\"odinger operator}
(with generalized edge weights) on a connected graph
$\Gamma=(\cV, \cE)$ in the following sense,
\begin{itemize}
\item $\cV = \{1,\ldots,N\}$,
\item $H_{u,v} = H_{v,u}$,
\item for $u\neq v$,
  \begin{equation*}
    H_{u,v} \neq 0 \quad \Leftrightarrow\quad (u,v) \in \cE.
  \end{equation*}
\end{itemize}

Let $T$ be a spanning tree of $\Gamma$ and let
$C = \cE(\Gamma) \setminus \cE(T)$.  There are exactly
$\beta = |\cE|-|\cV|+1$ edges in the set $C$.  We assume the graph $\Gamma$ is
not a tree itself, i.e.\ $\beta > 0$.

Orient each edge in $C$ in an arbitrary fashion and order the set $C$.
Let $\vec\alpha$ be a point in the $\beta$-dimensional torus
$\mathbb{T}^\beta := (-\pi,\pi]^\beta$ and denote by $H(\vec\alpha)$
the \term{magnetic Schr\"odinger operator} obtained from $H$ by
letting
\begin{equation}
  \label{eq:magnetic}
  H(\vec\alpha)_{u_j, v_j} = e^{i\alpha_j} H_{u_j, v_j},
  \qquad H(\vec\alpha)_{v_j, u_j} = e^{-i\alpha_j} H_{v_j, u_j},
\end{equation}
if $(u_j,v_j) \in C$ and $H(\vec\alpha)_{u,v} = H_{u.v}$ otherwise.
We note that $H(0) = H$.

\newcommand{\ao}{\vec\alpha^\circ}
\begin{theorem}[And extended version of \cite{Ber_apde13,Col_apde13}]
  \label{thm:mag_nodal}
  Let $\ao \in \{0,\pi\}^\beta$, let $\lo$ be the $n$-th eigenvalue in
  the spectrum of $H\left(\ao\right)$.  Assume $\lo$ is simple and the
  corresponding eigenvector $f$ has no zero entries.
  Consider $\Lambda(\vec\alpha)$, the smooth continuation of the eigenvalue
  $\lo$ in the spectrum of $H(\vec\alpha)$.  Then
  \begin{enumerate}
  \item $\Lambda(\vec\alpha)$ has a critical point $\vec\alpha=\ao$,
  \item the Morse index of the critical point is equal to the \term{nodal
    surplus} of $f$ defined as
    \begin{equation}
      \label{eq:nodal_surplus_def}
      \sigma = \phi(f,\Gamma) - (n-1), 
    \end{equation}
    where $\phi(f,\Gamma)$ is the \term{flip count} of $f$ with respect
    to the graph $\Gamma$,
    \begin{equation}
      \label{eq:flip_count_def}
      \phi(f,\Gamma)
      = \#\left\{(u,v)\in\cE \colon -f_u f_v H(\ao)_{u,v} < 0\right\}.
    \end{equation}
  \end{enumerate}
\end{theorem}

\begin{proof}
  For an $e = (v_1,v_2) \in C$, define
  \begin{equation*}
    s_e = \sign\left(-H(\ao)_{v_1,v_2}f_{v_1}f_{v_2}\right),
    \qquad
    p_e = \sqrt{\left|H(\ao)_{v_1,v_2} {f_{v_2}}/{f_{v_1}}\right|},
  \end{equation*}
  and introduce a $\beta\times|\cV|$ matrix $K(\vec\alpha)$
  \begin{equation*}
    K(\vec\alpha)_{e, v} =
    \begin{cases}
      p_e s_e
      & \mbox{if } v = v_1,\\
      e^{i(\alpha_e-\alpha_e^\circ)} H(\ao)_{v_1,v_2} / p_e
      & \mbox{if } v = v_2,\\
      0 & \mbox{otherwise},
    \end{cases}
    \qquad\mbox{where }
    e=(v_1,v_2) \in C.
  \end{equation*}
  A direct calculation shows that $K(\ao) f = \vec0$.

  Let $\Omega$ be the diagonal $\beta\times\beta$ matrix of signs
  $s_e$ and consider the matrix
  \begin{equation}
    \label{eq:Sdef}
    S := H(\vec\alpha) - K(\vec\alpha)^* \Omega K(\vec\alpha),
    \qquad \vec\alpha \in \R^\beta.
  \end{equation}
  The elements of $S$ corresponding to the edges $e\in C$ are zero;
  moreover the matrix $S$ is independent of $\vec\alpha$.  In other
  words, the matrix-function
  \begin{equation}
    \label{eq:HC}
    H^{\C}(\vec\alpha) := S + K(\vec\alpha)^* \Omega K(\vec\alpha),
    \qquad \vec\alpha \in \C^\beta
  \end{equation}
  coincides with $H(\vec\alpha)$ for real $\vec\alpha$.

  Consider the function
  $\Lambda^\C(\vec{\alpha}) = \lambda_n\big(H^\C(\vec\alpha)\big)$.
  By Theorems~\ref{thm:main_additive} and \ref{T:Curved}\footnote{One
    checks the isomorphism condition by calculating partial
    derivatives of $K(\vec\alpha)f$ with respect to $\alpha_j$ and
    noting that $f$ has no zero entries.} its Hessian at
  $\vec\alpha = \ao$ is the operator \eqref{eq:Hessian_operator0}
  which is a matrix with real entries.  Being real, it coincides with
  the Hessian of the function
  $\Lambda(\vec{\alpha}) = \lambda_n\big(H(\vec\alpha)\big)$ of the
  real argument.  Furthermore, its Morse index $\mu$ is equal to
  $m-n+\omega_-$, where $m$ is such that $\lo = \lambda_m(S)$ and
  $\omega_-$ is the number of $e\in C$ with $s_e<0$.

  The graph corresponding to the matrix $S$ is the spanning tree
  $T = \Gamma \setminus C$ we chose, and we have
  \begin{equation*}
    \phi(f, \Gamma) = \phi(f, T) + \omega_-.
  \end{equation*}
  Because the matrix $T$ is acyclic and the eigenfunction $f$ has no
  zero entries, the eigenvalue $\lo$ is simple in the spectrum of $S$,
  see \cite{Fie_cmj75}, and the critical point $\ao$ of the function
  $\Lambda(\vec\alpha)$ is non-degenerate.

  The same paper
  \cite{Fie_cmj75} also established that $\phi(f, T) = m-1$, where $m$
  is as above, i.e.\ the number of the $\lo$ in the spectrum of $S$.
  Combining all of the above, we get
  \begin{equation*}
    \mu = m-n+\omega_- = 1 + \phi(f, T) - n + \omega_-
    = \phi(f, \Gamma) - (n-1).
  \end{equation*}
\end{proof}

\section*{Acknowledgment}

An engaging joint research project with Yaiza Canzani, Graham Cox and
Jeremy Marzuola at an AIM-sponsored SQuaREs meeting has led the first
author to make conjectures that resulted in this paper.  Stimulating
discussions with Lior Alon, Ram Band, Yves Colin de Verdi\'ere, Alex
Elgart, Fritz Gesztesy, Dirk Hundertmark, Yuri Latushkin, Alxander
Pushnitski, Selim Sukhtaiev, Anna Vershinina and Igor Zelenko have
helped us along the way.  The authors are also grateful to the
reviewer for several improving suggestions.

The first author is partially supported by the NSF grant DMS--1815075.
The second author thanks NSF for the support from grants DMS--1517938
and DMS--2007408.

\bibliographystyle{myalpha}
\bibliography{bk_bibl,additional}

\end{document}